\documentclass{aims}
\usepackage{amsmath}
  \usepackage{paralist}
  \usepackage{graphics} 
  \usepackage{epsfig} 
\usepackage{graphicx}  
\usepackage{amsrefs}
\usepackage{epstopdf}
 \usepackage[colorlinks=true]{hyperref}
\hypersetup{urlcolor=blue, citecolor=red}

\def\currentvolume{X}
 \def\currentissue{X}
  \def\currentyear{200X}
   \def\currentmonth{XX}
    \def\ppages{X--XX}
     \def\DOI{10.3934/xx.xx.xx.xx}

\newtheorem{theorem}{Theorem}[section]

\newtheorem{lemma}[theorem]{Lemma}

\theoremstyle{definition}
\newtheorem{definition}[theorem]{Definition}
\newtheorem{remark}{Remark}

\newcommand{\eps}[1]{{#1}_{\varepsilon}}
\newcommand{\del}{\partial} 
\newcommand{\norm}[1]{\ensuremath{\left\| #1 \right\|}}
\newcommand{\ip}[2]{\ensuremath{\left\langle #1, #2 \right\rangle}}

\newcommand{\grad}{\nabla}

\def\eps{{\varepsilon}}
\def\bfR{\mathbb{R}}
\def\bfN{\mathbb{N}}

\title[Increased Parabolic Regularity] 
      {A Short Proof of Increased Parabolic Regularity}

\author[Stephen Pankavich and Nicholas Michalowski]{}


 \email{pankavic@mines.edu}
 \email{nmichalo@nmsu.edu}

\thanks{The first author is supported by NSF grant DMS-1211667.}

\begin{document}
\maketitle

\centerline{\scshape Stephen Pankavich }
\medskip
{\footnotesize
 \centerline{Department of Applied Mathematics and Statistics}
   \centerline{Colorado School of Mines}
   \centerline{Golden, Colorado 80401 USA}
} 

\medskip

\centerline{\scshape Nicholas Michalowski}
\medskip
{\footnotesize
 \centerline{Department of Mathematical Sciences}
   \centerline{New Mexico State University}
   \centerline{Las Cruces, New Mexico 88003 USA}
}

\bigskip

 \centerline{(Communicated by the associate editor name)}

\begin{abstract}
We present a new, short proof of the increased regularity obtained by solutions to uniformly parabolic partial differential equations.
Though this setting is fairly introductory, our new method of proof, which uses \emph{a priori} estimates, can be extended to prove analogous results for problems with time-dependent coefficients, transport equations, and nonlinear equations even when other tools, such as semigroup methods or the use of explicit fundamental solutions, are unavailable.
\end{abstract}

\section{Introduction}
It is well-known that solutions of uniformly parabolic partial differential equations possess a smoothing property.  That is, beginning with initial data which may fail to be even weakly differentiable, the solution becomes extremely smooth, gaining spatial derivatives at any time $t > 0$.
Though this property is well-established, such a result is often excluded from many standard texts in PDEs \cite{Evans, Folland, GT, Han, John, RR}.  
Of course, these works contain theorems demonstrating the regularity of solutions, but the same degree of regularity is assumed for the initial data. The notable exception is \cite[Thm 10.1]{Brezis} in which a gain of regularity theorem is proved, specifically for the heat equation with initial data in $L^2(\Omega)$ using semigroup methods. 
Even more concentrated works on the subject of parabolic PDE \cite{Friedman, Lieberman} do not contain such results regarding increased regularity of solutions to these equations. 

In the current paper, we will present a few results highlighting the increased regularity that solutions of these equations possess.
In particular, our method of proof is both new and brief, and relies on \emph{a priori} estimates. Hence, when traditional tools like semigroup methods or explicit fundamental solutions cannot be utilized, the new approach contained within may still be effective.

\section{Main Results}

Let $n \in \bfN$ be given.  We consider the Cauchy problem
\begin{equation}
\left \{
\begin{gathered}
\label{Par}
\partial_t u - \grad \cdot (D(x) \grad u) = f(x), \quad x \in \bfR^n, t>0\\
u(0,x) = u_0(x), \quad x \in \bfR^n
\end{gathered}
\right.
\end{equation}
where $D$, the diffusion matrix, and $f$, a forcing function, are both given.
Equations like (\ref{Par}) arise within countless applications as diffusion is of fundamental importance to physics, chemistry, and biology, especially for problems in thermodynamics, neuroscience, cell biology, and chemical kinetics. 
As we are interested in displaying the utility of our method of proof, we wish to keep the framework of the current problem
relatively straightforward.  
Thus, we will assume throughout that the diffusion matrix $D = D(x)$ satisfies the uniform ellipticity condition
\begin{equation}
\label{UE}
 w \cdot D(x) w \geq  \theta \vert w \vert^2
\end{equation}
for some $\theta > 0$ and all $x, w \in \bfR^n$.
We note that under suitable conditions on the spatial decay of $u$, our method may also be altered to allow for diffusion coefficients that are not \emph{uniformly} elliptic (see \cite{RVMFP}).
Additionally, we will impose different regularity assumptions on $D$ and $f$ to arrive at different conclusions regarding the regularity
of the solution $u$.

Throughout the paper we will only assume that the initial data $u_0$ is square integrable.
Hence, even though $u_0(x)$ may fail to possess even a single weak derivative, we will show that $u(t,x)$ gains spatial derivatives in $L^2$ on $(0,\infty) \times \bfR^n$.  Hence, by the Sobolev Embedding Theorem, solutions may be classically differentiable in $x$ assuming enough regularity of the coefficients.  In addition, we will show that $u$ is continuous in time at any instant after the initial time $t = 0$.
Though the setting (\ref{Par}) is fairly introductory and the assumptions on $D$ and $f$ are somewhat strong, the new method of proof can be adapted to extend the results to problems with time-dependent terms, transport equations, systems of parabolic PDEs, different spatial settings such as a bounded domain or manifold, and nonlinear equations, including nonlinear Fokker-Planck equations and nonlinear transport problems arising in Kinetic Theory \cite{NFP, VMFP, RVMFP}.

For the proofs, we will rely on \emph{a priori} estimation and the standard Galerkin approximation to obtain regularity
of the approximating sequence and then pass to the limit in order to obtain increased regularity of the solution.
Hence, we focus on deriving the appropriate estimates as the remaining machinery is standard (cf. \cite{Evans, GT}).
In what follows, $C >0$ will represent a constant that may change from line to line, and
for derivatives we will use the notation $$\Vert \grad_x^k u(t) \Vert_2^2 := \sum_{\vert \alpha \vert = k} \Vert \partial_x^\alpha u(t) \Vert_2^2$$
to sum over all multi-indices of order $k \in \bfN$.  When necessary, we will specify parameters on which constants may depend by using a subscript (e.g., $C_T$).

Our first result establishes the main idea for low regularity of $D$ and $f$.

\begin{theorem}[Lower-order Regularity]
\label{L1}
Assume $f \in H^1(\bfR^n)$, $D \in W^{1,\infty}(\bfR^n; \bfR^{n \times n})$,
and $u_0 \in L^2(\bfR^n)$.  Then, for any $T > 0$ and $t \in (0,T]$, any solution of (\ref{Par})
satisfies
$$\sup_{0\leq t \leq T}\norm{u(t)}_2^2\leq C_T (\Vert u_0
\Vert_2^2+\norm{f}_{H^1}^2)\quad\text{and}\quad\norm{\grad_xu(t)}_2^2 \leq \frac{C_T}{t}(\Vert u_0
\Vert_2^2+\norm{f}_{H^1}^2).$$
\end{theorem}

\begin{proof} 
We first prove two standard estimates of (\ref{Par}).  First, we multiply by $u$, integrate the equation in $x$, integrate by parts and use Cauchy's
Inequality to find
$$\frac{1}{2} \frac{d}{dt} \Vert u(t) \Vert_2^2 + \int_{\bfR^n}\grad_x u(t) \cdot D \grad_x u(t) dx \  \leq \frac{1}{2} \left ( \Vert f \Vert_2^2 + \Vert u(t) \Vert_2^2 \right ).$$
Then, using (\ref{UE}) and the regularity assumption on $f$, we find
\begin{equation}
\label{L2}
\frac{1}{2} \frac{d}{dt} \Vert u(t) \Vert_2^2  \leq C \left (\norm{f}_{H^1}^2 + \Vert u(t) \Vert_2^2 \right ) - \theta \Vert \grad_x u(t) \Vert_2^2.
\end{equation}
Next, we take any first-order derivative with respect to $x$ (denoted by $\partial_x$) of the equation, multiply by $\partial_x u$, and integrate to obtain
$$ \frac{1}{2} \frac{d}{dt} \Vert \partial_x u(t) \Vert_2^2 +\int_{\bfR^n}\grad_x \partial_x u(t) \cdot D \grad_x \partial_x u(t) dx \ + \int_{\bfR^n} \grad_x \partial_x u \cdot \partial_x D \grad_x u \ dx 
= \int_{\bfR^n} \partial_x f \partial_x u \ dx.$$
Thus, using Cauchy's inequality with the ellipticity and regularity assumptions, we find for any $\eps > 0$
\begin{align*}
\frac{1}{2} \frac{d}{dt} \Vert \partial_x u(t) \Vert_2^2 & \leq   - \int_{\bfR^n}\grad_x \partial_x u(t) \cdot D \grad_x \partial_x u(t) dx - \int_{\bfR^n} \grad_x \partial_x u \cdot \partial_x D \grad_x u \ dx 
+ \int_{\bfR^n} \partial_x f \partial_x u \ dx\\
& \leq   - \theta \Vert \grad_x \partial_x u(t) \Vert_2^2 + \Vert D \Vert_{W^{1,\infty}} \left ( \eps \Vert \grad_x \partial_x u(t) \Vert_2^2 + \frac{1}{\eps} \Vert \grad_x u(t) \Vert_2^2 \right )\\
&  \quad  +  \frac{1}{2} \left ( \Vert \partial_x f \Vert_2^2 + \Vert \partial_x u(t) \Vert_2^2 \right ).
\end{align*}
Choosing $\eps = \theta(2\Vert D \Vert_{W^{1,\infty}})^{-1}$ and summing over all first-order spatial derivatives, we finally arrive at the estimate
\begin{equation}
\label{H1}
\frac{1}{2} \frac{d}{dt} \Vert \grad_x u(t) \Vert_2^2 
 \leq C \left (\norm{f}_{H^1}^2+ \Vert \grad_x u(t) \Vert_2^2 \right )
 - \frac{\theta}{2} \Vert \grad_x^2 u(t) \Vert_2^2.
\end{equation}

Now, we utilize a linear expansion in $t$ to prove the theorem.  Let $T > 0$ be given. Consider $t \in (0,T]$ and define
$$M_1(t) = \Vert u(t) \Vert_2^2 + \frac{\theta t}{2} \Vert \grad_x u(t) \Vert_2^2.$$
We differentiate this quantity, and use the estimates (\ref{L2}) and (\ref{H1}) to find
\begin{eqnarray*}
M_1'(t) & = & \frac{d}{dt} \Vert u(t) \Vert_2^2 + \frac{\theta}{2} \Vert \grad_x u(t) \Vert_2^2 +  \frac{\theta t}{2} \frac{d}{dt} \Vert \grad_x u(t) \Vert_2^2 \\
& \leq & C \left (\norm{f}_{H^1}^2 + \Vert u(t) \Vert_2^2 \right ) - 2\theta \Vert \grad_x u(t) \Vert_2^2 + \frac{\theta}{2} \Vert \grad_x u(t) \Vert_2^2 \\
& \ & \ \ +  \frac{\theta t}{2} \left [  2C \left (\norm{f}_{H^1}^2+ \Vert \grad_x u(t) \Vert_2^2 \right )
 - \theta \Vert \grad_x^2 u(t) \Vert_2^2 \right ]\\
& \leq & C_T\left (\norm{f}_{H^1}^2 + M_1(t) \right )
\end{eqnarray*}
A straightforward application of Gronwall's inequality (cf. \cite{Evans}) then implies
$$ M_1(t) \leq C_T (M_1(0)+\norm{f}_{H^1}^2) = C_T (\Vert u_0 \Vert_2^2+\norm{f}_{H^1}^2).$$
Finally, the bound on $M_1(t)$ yields
$$ \Vert u(t) \Vert_2^2 \leq C_T, \quad \Vert \grad_x u(t) \Vert_2^2 \leq \frac{C_T}{\theta t}$$
and the estimate holds on the interval $(0,T]$.  As $T > 0$ is arbitrary, the result follows.

%
\end{proof}

Next, we formulate the existence of weak solutions for our lower-order
regularity setting.

\begin{definition}
We say that $u \in L^2([0,T];H^1(\bfR^n))$, with $\del_t u \in L^2([0,T];H^{-1}(\bfR^n)$ is a weak solution of \eqref{Par} if 
\begin{equation}
\left \{
\begin{gathered}
\label{weak_Par}
\ip{\del_t u}{v} +\ip{D(x) \grad u}{\grad v} = \ip{f}{v} \\
u(0,x) = u_0(x), \quad x \in \bfR^n
\end{gathered}
\right.
\end{equation}
for every $v \in H^1(\bfR^n)$ and $t\in [0,T]$.
\end{definition}

\begin{theorem}[Existence and Uniqueness of Weak solutions]\label{weak_exist}
   Given any $u_0\in L^2(\bfR^n)$ and $f\in H^1(\bfR^n)$ and $T>0$
   arbitrary, there exists a unique $u \in C((0,T]; H^1(\bfR^n))
   \cap C([0,T]; L^2(\bfR^n))$ and $u'\in L^2(([0,T]; H^{-1}(\bfR^n))$
   that solves \eqref{weak_Par}.
\end{theorem}
\begin{proof}
  We follow a standard Galerkin approach. Take
  $\{w_k(x)\}_{k=0}^\infty$ to be an orthonormal basis for $L^2$ with
  $w_k\in H^s$ for $s\geq 0$. 
  Consider functions of the form $u_m(x,t)=\sum_{k=0}^m d_k^m(t)
  w_k(x)$ with $d_k^m(t)$ a smooth function of $t$.  Then the equations
  \begin{gather*}
  \ip{\del_t u_m}{w_k} +\ip{D(x) \grad u_m}{\grad w_k} = \ip{f}{w_k} \\
      \ip{u_m(0,\cdot)}{w_k} = \ip{u_0}{w_k}
  \end{gather*}
  for $k=1,2,\ldots m$ reduce to a constant coefficient first order
  system of ODE's for $d_k^m(t)$, and hence existence of approximate
  solutions is readily established.

  For these solutions, $u_m(t)$, we may repeat the proof of our \emph{a priori}
  estimates verbatim.  Thus we can conclude that
  $$\sup_{0\leq t\leq T} \Big(\norm{u_m(t)}_2^2+\frac{\theta
    t}{2}\norm{\grad_x u_m(t)}_2^2 \Big)\leq C_T\norm{u_0}_2^2.$$
  From the proof, we also have the inequality
  $$\int_0^T\frac{1}{2}\frac{d}{dt}\norm{u_m(t)}_2^2\,dt+\int_0^T\theta\norm{\grad_x u_m}_2^2\,dt\leq
  C_T\int_0^T\Big(\norm{f}_2^2+\norm{u_0}_2^2\Big)\,dt.$$
  Using the above control of $\sup_{0\leq t \leq T} \norm{u_m(t)}_2^2$,
  we find that $$\int_0^T\norm{u_m}_{H^1(\bfR^n)}^2\leq
  C_T\Big(\norm{f}_2^2+\norm{u_0}_2^2\Big).$$
  Finally, fix $v\in H^1(\bfR^n)$ with $\norm{v}_{H^1}\leq 1$ and
  consider $$\ip{\partial_t u_m}{v}=-\ip{D(x)\grad_x
    u_m}{\grad_x v}+\ip{f}{v}.$$ 
  Using Cauchy-Schwartz and taking the supremum over $v\in H^1$ with
  $\norm{v}_{H^1}\leq 1$ we find
  $\norm{\del_t u_m(t)}_{H^{-1}}\leq
  C_T\Big(\norm{f}_2+\norm{u_0}_2\Big)$ and 
  $$\int_0^T\norm{\del_t u_m (t)}_{H^{-1}}^2 \ dt \leq
  C_T\Big(\norm{f}_2^2+\norm{u_0}_2^2\Big).$$
  Thus, $u_m$ is a bounded sequence in $L^2([0,T];H^1(\bfR^n))$ and $\del_t
  u_m$ is a bounded sequence in $L^2([0,T];H^{-1}(\bfR^n)),$ so we may extract
  a subsequence ${m_j}$ so that $u_{m_j}\to u$ in $L^2([0,T];H^1(\bfR^n))$ and
  $\del_tu_{m_j}\to \del_t u$ in $L^2([0,T];H^{-1}(\bfR^n))$.
  
  Now fix an integer $N$ and consider $v(t)=\sum_{k=0}^N
  d_k(t)w_k(x)$, where $d_k(t)$ are fixed smooth functions.  Then for
  $m_j>N$, we have
  $$\int_0^T \ip{\del_t u_{m_j}}{v} +\ip{D(x) \grad u_{m_j}}{\grad v}\,dt = \int_0^T
  \ip{f}{v}\,dt.$$ 
  Thus passing to the limit 
  $$\int_0^T \ip{\del_t u}{v} +\ip{D(x) \grad u}{\grad v}\,dt = \int_0^T
  \ip{f}{v}\,dt.$$ 
  Since $v$ given above are dense in
  $L^2([0,T];H^1(\bfR^n))$, the equality holds for any $v$ in
  this space.  Since $u\in L^2([0,T];H^1(\bfR^n))$ and $\del_t u \in
  L^2([0,T];H^{-1}(\bfR^n))$ we have that $u\in C([0,T];L^2(\bfR^n)$
  by \cite[Thm.~3~\S 5.9.2]{Evans}.

  To see $u(0)=u_0$, we take $v\in C^1([0,T];H^1(\bfR^n))$ with the
  property that $v(T)=0$, then we find that 
  $$\ip{u_{m_j}(0)}{v(0)}-\int_0^T \ip{u_{m_j}}{\del_t v} +\ip{D(x) \grad u_{m_j}}{v}\,dt = \int_0^T
  \ip{f}{v}\,dt.$$ 
  Notice $u_{m_j}(0)=\sum_{k=0}^{m_j} \ip{u_0}{w_k} w_k \to u_0$ in
  $L^2(\bfR^n)$ as $m_j\to \infty$.  Thus passing to the limit, we
  find that $\ip{u(0)}{v(0)}=\ip{u_0}{v(0)}$ for $v(0)$ arbitrary.
  Hence $u(0)=u_0$.

  To prove uniqueness, notice that for any two
  solutions $u$ and $\tilde u$ the difference $u-\tilde u$ satisfies
  our equation with $u_0=0$ and $f=0$.  Thus our a priori estimate
  gives that $\sup_{0\leq t\leq T}\norm{u(t)-\tilde u(t)}_2^2\leq 0$
  and uniqueness follows immediately.
  
  Finally, to show that $u\in C((0,T];H^1(\bfR^n))$ we consider
  $w_s(t)=u(t+s)-u(t)$.  Then $w_s(t)$ satisfies our equation with
  $f=0$ and $w(0)=u_0-u(s)$.  From our \emph{a priori} estimate
  $$\norm{w_s(t)}+\frac{\theta t}{2}\norm{\grad_x w_s(t)}_2^2 \leq
    C_T(\norm{u_0-u(s)}_2^2).$$
  From the fact that $u\in C([0,T];L^2(\bfR^n))$, we have for $t>0$ that
  $\lim_{s\to 0} \norm{u(t+s)-u(t)}=0$ and $\lim_{s\to 0}\norm{\grad_x
  u(t+s)-\grad_xu(t)}=0$, whence the result follows.
\end{proof}

Next, we extend the previous estimate to higher regularity assuming that $D$ and $f$ possess additional weak derivatives.

\begin{lemma}[Higher-order Regularity]
\label{T1}
For every $m \in \bfN$, if $f \in H^m(\bfR^n)$, $D \in W^{m,\infty}(\bfR^n; \bfR^{n \times n})$,
and $u_0 \in L^2(\bfR^n)$, then the previously derived solution of
(\ref{Par}) satisfies 
$$\norm{\grad_x^k u}_2^2\leq \frac{C_T}{t^k}\left( \norm{u_0}_2^2 +
  \norm{f}_{H^k}^2\right)\quad \text{for } k=0, 1, \ldots, m.$$
\end{lemma}

\begin{proof} 
We will prove the result by induction on $m$.  The base case ($m=1$) follows immediately from Lemma \ref{L1}.
Prior to the inductive step, we first prove a useful estimate for solutions of (\ref{Par}).  
For the estimate, assume $D$ and $f$ possess $k \in \mathbb{N}$ derivatives in $L^\infty$ and $L^2$, respectively.
Take any $k$th-order derivative with respect to $x$ (denoted by $\partial^\alpha_x$) of the equation, multiply by $\partial^\alpha_x u$, and integrate using integration by parts to obtain
$$ \frac{1}{2} \frac{d}{dt} \Vert \partial^\alpha_x u(t) \Vert_2^2 +
\int_{\bfR^n} \grad_x \partial_x^\alpha u \cdot
\sum_{j=0}^k\sum_{\substack{|\beta|=j\\ \beta+\gamma=\alpha}} {\alpha \choose \beta} \partial_x^\beta D \grad_x \partial_x^{\gamma} u \ dx 
= \int_{\bfR^n} \partial_x^\alpha f \partial_x^\alpha u \ dx$$
and thus
$$  \frac{1}{2} \frac{d}{dt} \Vert \partial^\alpha_x u(t) \Vert_2^2 = - \int_{\bfR^n} \grad_x \partial_x^\alpha u \cdot \sum_{j=0}^k\sum_{\substack{|\beta|=j\\ \beta+\gamma=\alpha}} {\alpha \choose \beta} \partial_x^\beta D \grad_x \partial_x^{\gamma} u \ dx +  \int_{\bfR^n} \partial_x^\alpha f \partial_x^\alpha u \ dx.$$

Labeling the first term on the right side $A$, we use (\ref{UE}) and the regularity of $D$ with Cauchy's inequality (with $\eps > 0$) to find
\begin{eqnarray*}
A & = & -  \int_{\bfR^n} \grad_x \partial_x^\alpha u \cdot D \grad_x \partial_x^\alpha u \ dx -  \int_{\bfR^n} \grad_x \partial_x^k u \cdot \sum_{j=1}^k\sum_{\substack{|\beta|=j\\ \beta+\gamma=\alpha}} {\alpha \choose \beta} \partial_x^\beta D \grad_x \partial_x^{\gamma} u \ dx \\
& \leq & - \theta \Vert  \grad_x \partial_x^\alpha u(t) \Vert_2^2 + C \Vert D \Vert_{W^{k,\infty}}^2 \left ( \eps \Vert  \grad_x \partial_x^\alpha u(t) \Vert_2^2 + \frac{1}{\eps} \Vert u(t) \Vert_{H^{k-1}}^2 \right )\\
& \leq & - \frac{\theta}2 \Vert  \grad_x \partial_x^\alpha u(t) \Vert_2^2 + C \Vert u(t) \Vert_{H^{k-1}}^2
\end{eqnarray*}
where we have chosen $\eps = \theta \left (2C\Vert D \Vert_{W^{k,\infty}}^2 \right)^{-1}$ in the third line.
Inserting this into the above equality and using Cauchy's inequality again we find
$$  \frac{1}{2} \frac{d}{dt} \Vert \partial^\alpha_x u(t) \Vert_2^2 \leq - \frac{\theta}{2} \Vert \grad_x \partial^\alpha_x u(t) \Vert_2^2  + C \Vert u(t) \Vert_{H^{k-1}}^2 +  \frac{1}{2} \Vert \partial^\alpha_x f \Vert_2^2  +
\frac{1}{2}  \Vert \partial_x^\alpha u(t) \Vert_2^2.$$
Finally, summing over all first-order derivatives and using the regularity of $f$ yields the estimate
\begin{equation}
\label{IndEst}
\frac{1}{2} \frac{d}{dt} \Vert \grad^k_x u(t) \Vert_2^2 \leq - \frac{\theta}{2} \Vert \grad^{k+1}_x u(t) \Vert_2^2 + C \left (\norm{f}_{H^k}^2 + \Vert u(t) \Vert_{H^k}^2 \right ).
\end{equation}

Now, we prove the lemma utilizing this estimate for $k = 0, 1, ..,m$. Assume $f \in H^m(\bfR^n)$ and $D \in W^{m,\infty}(\bfR^n; \bfR^{n \times n})$.  Since this implies that
$f \in H^{m-1}(\bfR^n)$ and $D \in W^{m-1,\infty}(\bfR^n; \bfR^{n \times n})$, we find that 
$u \in C((0,\infty); H^{m-1}(\bfR^n))$ by the induction hypothesis. Let $T>0$ be given.  Consider $t \in (0,T]$ and define
$$M(t) = \sum_{k=0}^m \frac{(\theta t)^k}{2^kk!} \Vert \grad^k_x u(t) \Vert_2^2.$$
We differentiate to find
$$M'(t) = \sum_{k=1}^m \frac{\theta^k t^{k-1}}{2^k(k-1)!} \Vert \grad^k_x u(t) \Vert_2^2 +  \sum_{k=0}^m \frac{(\theta t)^k}{2^kk!} \frac{d}{dt}\Vert \grad^k_x u(t) \Vert_2^2
=: I + II.$$
We use (\ref{IndEst}) for any $k = 0, ..., m$ and relabel the index of the sum so that
\begin{eqnarray*}
II & \leq & \sum_{k=0}^m \frac{(\theta t)^k}{2^kk!} \left ( -\theta \Vert \grad^{k+1}_x u(t) \Vert_2^2 + C \left (\norm{f}_{H^k}^2 + \Vert u(t) \Vert_{H^k}^2 \right ) \right ) \\
& = & -2\sum_{k=0}^m \frac{\theta^{k+1} t^k}{2^{k+1} k!} \Vert \grad^{k+1}_x u(t) \Vert_2^2 + C \sum_{k=0}^m \frac{(\theta t)^k}{2^kk!} \left (\norm{f}_{H^k}^2 + \Vert u(t) \Vert_{H^k}^2 \right )  \\
& \leq & - 2I - \frac{\theta^{m+1} t^m}{2^{m+1}m!}\norm{\grad_x^{m+1}u(t)}_2^2 + C \sum_{k=0}^m \frac{(\theta t)^k}{2^kk!} \left (\norm{f}_{H^k}^2 + \Vert u(t) \Vert_{H^k}^2 \right ) 
\end{eqnarray*}
Notice the induction hypothesis gives $\norm{u(t)}_{H^{k-1}}^2\leq
\frac{C_T}{t^{k-1}}(\norm{f}_{H^{k-1}}^2+\norm{u_0}_2^2$).  Using this bound and the
previous inequality within the estimate of $M'(t)$, we find
\begin{eqnarray*}
M'(t) & \leq & C \sum_{k=0}^m \frac{(\theta t)^k}{2^kk!} \left (\norm{f}_{H^k}^2 + \Vert u(t) \Vert_{H^k}^2 \right )\\
& \leq & C_T \left ( \norm{f}_{H^m}^2 +  \sum_{k=0}^m \frac{(\theta t)^k}{2^kk!} \left [ \Vert \grad_x^k u(t) \Vert_{2}^2 + \Vert u(t) \Vert_{H^{k-1}}^2 \right ] \right )\\
& \leq & C_T \left ( \norm{f}_{H^m}^2 + M(t) + \sum_{k=0}^m \frac{(\theta t)^k}{2^k k!} \frac{C_T}{t^{k-1}} \left ( \Vert f \Vert_{H^{k-1}}^2 + \Vert u_0 \Vert_2^2 \right ) \right ) \\
& \leq & C_T \left ( \norm{f}_{H^m}^2+\norm{u_0}_2^2 + M(t) \right ).  \\
\end{eqnarray*}
Another straightforward application of Gronwall's inequality then implies
$$ M(t) \leq C_T \Big(\norm{f}_{H^m}^2+\norm{u_0}_2^2+M(0)\Big) \leq
C_T(\norm{f}_{H^m}^2 +\Vert u_0 \Vert_2^2).$$
Finally, the bound on $M(t)$ yields
$$ \Vert \grad^m_x u(t) \Vert_2^2 \leq \frac{C_T}{(\theta t)^m}$$
which completes the inductive step and the proof of the lemma.
\end{proof}

\begin{theorem}[Existence and Uniqueness of Weak solutions]
   Let $m \in \mathbb{N}$ be given. For any $u_0\in L^2(\bfR^n)$, $f\in H^m(\bfR^n)$, and $T>0$
   arbitrary, there exists a unique $u \in C((0,T]; H^m(\bfR^n))
   \cap C([0,T]; L^2(\bfR^n))$ and $u'\in L^2(([0,T]; H^{-1}(\bfR^n))$
   that solves \eqref{weak_Par}.
\end{theorem}
\begin{proof}
  This follows by a straightforward repetition of the proof of
  Theorem~\ref{weak_exist} with the obvious modifications.
\end{proof}

Of course, if the dimension $n$ satisfies $n < 2m -1$ this result implies classical differentiability of solutions and they satisfy the PDE in the classical sense.
Finally, this result can be easily used to deduce infinite spatial differentiability of the solution assuming $D$ and $f$ satisfy the same
condition.

\begin{theorem}[Infinite Differentiability]
\label{T2}
If $f \in H^\infty(\bfR^n)$, $D \in W^{\infty,\infty}(\bfR^n; \bfR^{n \times n})$,
and $u_0 \in L^2(\bfR^n)$, then for any $T > 0$ arbitrary, any solution of (\ref{Par}) satisfies $u \in C^\infty((0,T] \times \bfR^n)$.
\end{theorem}

\begin{remark}
On a bounded domain, it is enough to impose $f \in C^\infty(\bfR^n)$ and $D \in C^\infty(\bfR^n; \bfR^{n \times n})$ to arrive at the same result.
\end{remark}

\begin{proof}
The result follows immediately by applying Theorem \ref{T1} for each $m \in \bfN$, noticing that 
$$ \partial_t u = \nabla \cdot (D \nabla u) + f$$ is continuous, and bootstrapping this property for higher-order time derivatives.
\end{proof}

\begin{remark}
Though we have chosen to demonstrate the method for equations with time-independent coefficients, the same results can be obtained for time-dependent diffusion coefficients $D$ and sources $f$ using the same proof, as long as these functions are sufficiently smooth in $t$.  Additionally, similar arguments can be used to gain regularity of the solution in $t$, as well.  
\end{remark}

\bibliographystyle{AIMS}

\begin{bibdiv}
\begin{biblist}

\bib{Brezis}{book}{
   author={Brezis, Haim},
   title={Functional analysis, Sobolev spaces and partial differential
   equations},
   series={Universitext},
   publisher={Springer},
   place={New York},
   date={2011},
   pages={xiv+599},
   isbn={978-0-387-70913-0},
   review={\MR{2759829 (2012a:35002)}},
}

\bib{Evans}{book}{
   author={Evans, Lawrence C.},
   title={Partial differential equations},
   series={Graduate Studies in Mathematics},
   volume={19},
   edition={2},
   publisher={American Mathematical Society},
   place={Providence, RI},
   date={2010},
   pages={xxii+749},
   isbn={978-0-8218-4974-3},
   review={\MR{2597943 (2011c:35002)}},
}



\bib{Folland}{book}{
   author={Folland, Gerald B.},
   title={Introduction to partial differential equations},
   edition={2},
   publisher={Princeton University Press},
   place={Princeton, NJ},
   date={1995},
   pages={xii+324},
   isbn={0-691-04361-2},
   review={\MR{1357411 (96h:35001)}},
}

\bib{Friedman}{book}{
   author={Friedman, Avner},
   title={Partial differential equations of parabolic type},
   publisher={Prentice-Hall Inc.},
   place={Englewood Cliffs, N.J.},
   date={1964},
   pages={xiv+347},
   review={\MR{0181836 (31 \#6062)}},
}
	
	
\bib{GT}{book}{
   author={Gilbarg, David},
   author={Trudinger, Neil S.},
   title={Elliptic partial differential equations of second order},
   series={Classics in Mathematics},
   note={Reprint of the 1998 edition},
   publisher={Springer-Verlag},
   place={Berlin},
   date={2001},
   pages={xiv+517},
   isbn={3-540-41160-7},
   review={\MR{1814364 (2001k:35004)}},
}
	
	
\bib{Han}{book}{
   author={Han, Qing},
   title={A basic course in partial differential equations},
   series={Graduate Studies in Mathematics},
   volume={120},
   publisher={American Mathematical Society},
   place={Providence, RI},
   date={2011},
   pages={x+293},
   isbn={978-0-8218-5255-2},
   review={\MR{2779549 (2012j:35001)}},
}	
	
\bib{John}{book}{
   author={John, Fritz},
   title={Partial differential equations},
   series={Applied Mathematical Sciences},
   volume={1},
   edition={4},
   publisher={Springer-Verlag},
   place={New York},
   date={1991},
   pages={x+249},
   isbn={0-387-90609-6},
   review={\MR{1185075 (93f:35001)}},
}


\bib{Lieberman}{book}{
   author={Lieberman, Gary M.},
   title={Second order parabolic differential equations},
   publisher={World Scientific Publishing Co. Inc.},
   place={River Edge, NJ},
   date={1996},
   pages={xii+439},
   isbn={981-02-2883-X},
   review={\MR{1465184 (98k:35003)}},
}

\bib{NFP}{article}{
   author={Alcantara Felix, Jose Antonio},
   author={Calogero, Simone},
   author={Pankavich, Stephen},
   title={Spatially homogeneous solutions of the Vlasov-Nordstrom-Fokker-Planck system},
   date={2014},
   pages={submitted for publication}
}




\bib{RVMFP}{article}{
   author={Michalowski, Nicholas},
   author={Pankavich, Stephen},
   title={Global existence for the one and one-half dimensional
   relativistic Vlasov-Maxwell-Fokker-Planck system},
   date={2014},
   pages={submitted for publication}
}


\bib{RR}{book}{
   author={Renardy, Michael},
   author={Rogers, Robert C.},
   title={An introduction to partial differential equations},
   series={Texts in Applied Mathematics},
   volume={13},
   edition={2},
   publisher={Springer-Verlag},
   place={New York},
   date={2004},
   pages={xiv+434},
   isbn={0-387-00444-0},
   review={\MR{2028503 (2004j:35001)}},
}

\bib{VMFP}{article}{
   author={Calogero, Simone},
   author={Pankavich, Stephen},
   title={Classical Solutions to the one and one-half dimensional Vlasov-Maxwell-Fokker-Planck system},
   date={2014},
   pages={submitted for publication}
}



\end{biblist}
\end{bibdiv}


\medskip
Received xxxx 20xx; revised xxxx 20xx.
\medskip

\end{document}